\documentclass[11pt]{amsart}

\usepackage[left=1in, right=1in, top=1.2in, bottom=1.2in]{geometry}
\usepackage{microtype, mathtools, mathrsfs, enumerate, dsfont, esvect, tikz}
\usepackage{verbatim}
\usepackage[linktocpage]{hyperref}
\usepackage[linktocpage]{hyperref}
\hypersetup{
 colorlinks=true,  
 linkcolor=red,   
 citecolor=blue,   
 filecolor=magenta,  
 urlcolor=cyan   
}

\usepackage{amssymb}
\usepackage{url}

\newtheorem{theorem}{Theorem}[section]

\newtheorem{lemma}[theorem]{Lemma}

\newtheorem{corollary}[theorem]{Corollary}

\theoremstyle{definition}
\newtheorem{question}[theorem]{Question}
\newtheorem{definition}[theorem]{Definition}

\theoremstyle{plain}
\numberwithin{equation}{theorem}

\theoremstyle{remark}

\newif\ifhascomments \hascommentstrue
\ifhascomments
 \newcommand{\dragos}[1]{{\color{red}[[\ensuremath{\bigstar\bigstar\bigstar} #1]]}}
 \newcommand{\matt}[1]{{\color{red}[[\ensuremath{\spadesuit\spadesuit\spadesuit} #1]]}}
\else
 \newcommand{\dragos}[1]{}
 \newcommand{\matt}[1]{}
\fi

\begin{document}

\title{Hopf Ore Extensions}

\author{Hongdi Huang}

\address{Department of Pure Mathematics\\
University of Waterloo\\
Waterloo, ON N2L 3G1\\
Canada}
\email{h237huan@uwaterloo.ca}

\begin{abstract}
Brown, O'Hagan, Zhang, and Zhuang gave a set of conditions on an automorphism $\sigma$ and a $\sigma$-derivation $\delta$ of a Hopf $k$-algebra $R$ for when the skew polynomial extension $T=R[x, \sigma, \delta]$ of $R$ admits a Hopf algebra structure that is compatible with that of $R$. In fact, they gave a complete characterization of which $\sigma$ and $\delta$ can occur under the hypothesis that $\Delta(x)=a\otimes x +x\otimes b +v(x\otimes x) +w$, with $a, b\in R$ and $v, w\in R\otimes_k R$, where $\Delta: R\to R\otimes_k R$ is the comultiplication map. In this paper, we show that after a change of variables one can in fact assume that $\Delta(x)=\beta^{-1}\otimes x +x\otimes 1 +w$, with $\beta $ is a grouplike element in $R$ and $w\in R\otimes_k R,$ when $R\otimes_k R$ is a domain and $R$ is noetherian. In particular, this completely characterizes skew polynomial extensions of a Hopf algebra that admit a Hopf structure extending that of the ring of coefficients under these hypotheses.  We show that the hypotheses hold for domains $R$ that are noetherian cocommutative Hopf algebras of finite Gelfand-Kirillov dimension.
\end{abstract}



\thanks{The author acknowledges support from the National Sciences and Engineering Research Council of Canada.}

\maketitle

\section{Introduction}
Throughout this paper, we take $k$ to be a field and all algebras are over $k$. Given a $k$-algebra $R$, a $k$-algebra automorphism $\sigma$ of $R$, and a $k$-linear $\sigma$-derivation $\delta$ of $R$ that is, $\delta(ab)=\sigma(a)\delta(b) +\delta(a)b$ for $a,b\in R$, one can form a skew polynomial extension $T=R[x, \sigma, \delta]$, which is the $k$-algebra generated by $R$ and the indeterminate $x$, subject to the relations $xr-\sigma(r)x=\delta(r)$ for all $r\in R$. Let $R=(R, m, \mu, \Delta, \epsilon, S)$ be a Hopf algebra with multiplication $m$, comultiplication $\Delta$, unit $\mu$, counit $\epsilon$ and antipode $S$. One can refer to the excellent books \cite{LS,M} for further background on Hopf algebras.
Panov \cite{Pa} asked the following natural question: Given a Hopf algebra $R$, for which automorphisms $\sigma$ and $\sigma$-derivations $\delta$ does the skew polynomial ring $T=R[x, \sigma, \delta]$ have a Hopf algebra structure extending the given structure on $R$? Panov answered this question under this hypothesis that $x$ is a skew primitive element of $T$; i.e., $\Delta(x)=x\otimes a_1+a_2\otimes x$, with $a_1, a_2\in R$. This does not give a complete answer, however, as there are examples of Hopf algebra extensions where this hypothesis does not hold. In particular, Brown, O'Hagan, Zhang, and Zhuang \cite{Gzhuang} gave such an example; namely, if we take $H$ to be the coordinate ring of the group $G$ of upper-triangular $3\times 3$ unipotent complex matrices then $H$ is generated as a $\mathbb{C}$-algebra by the coordinate functions $x,y,z$, where evaluating $x, y$ and $ z$ at an element of $G$ corresponds to taking respectively the $(1, 3)$-, $(1, 2)$-, and $(2, 3)$-entry of the element. Then $H=k[y, z][x]$ with coefficient Hopf algebra $R=k[y, z]$, but $x$ is not skew primitive, since a straightforward computation shows that $\Delta(x)=x\otimes 1+1\otimes x+y\otimes z$. To deal with such examples, \cite{Gzhuang} relaxed Panov's hypothesis and studied skew polynomial extensions of Hopf algebras in which $\Delta(x)$ is of the form $s\otimes x+ x\otimes t+ v(x\otimes x)+w$, where $s, t\in R$ and $v, w\in R\otimes_k R$. In addition, Brown et al. \cite{Gzhuang} gave the following definition. 
\begin{definition}
\label{defn:BOZZ}
Let $R$ be a Hopf $k$-algebra. A \emph{Hopf Ore Extension} (HOE) of $R$ is a $k$-algebra $T$ such that: 
\begin{enumerate}
 \item $T$ is a Hopf $k$-algebra with Hopf subalgebra $R$;
 \item there exists an algebra automorphism $\sigma$ and a $\sigma$-derivation $\delta$ of $R$ such that $T=R[x, \sigma, \delta]$;
 \item there are $a, b\in R$ and $w, v\in R\otimes_k R$ such that
 $\Delta(x)=a\otimes x+ x\otimes b+v(x\otimes x)+w.$
\end{enumerate} 
\end{definition}
Using this definition \cite{Gzhuang} gave a complete list of conditions on $\sigma$ and $\delta$ for when $T$ is a HOE of $R$. In addition, they asked the following question.
\begin{question}
\label{quest} Does the third condition in Definition \ref{defn:BOZZ} follow from the first two, after a change of the variable $x$? 
\end{question}
This is an important question, as it is unclear in general whether a skew polynomial extension of a Hopf algebra $R$ that is itself a Hopf algebra should have the property that the indeterminate $x$ is so well-behaved under the comultiplication map.

If one could answer their question, one can remove this final hypothesis and give a complete answer to Panov's original question. We note that in the case when $R$ is a connected Hopf algebra (that is, a Hopf algebra whose coradical is the base field), Brown, O'Hagan, Zhang, and Zhuang \cite{Gzhuang} showed that the answer to their question is affirmative and that after a change of variables one can have $\Delta (x)=1\otimes x + x\otimes 1+ w$ for some $w\in R\otimes_k R$. Recently, it was shown in the case that $R$ is a commutative affine Hopf $k$-algebra and $k$ is algebraically closed and of characteristic zero then after a change of variables one has $\Delta(x)=s\otimes x+ x\otimes t+ w$ \cite{Bell}. This, then, gives another instance of when the question of Brown, O'Hagan, Zhang, and Zhuang \cite{Gzhuang} has a positive answer. In this paper, we will show that under the hypotheses that $R\otimes_k R$ is a domain and $R$ is noetherian then after a change of variables we have $\Delta(x)=\beta^{-1}  \otimes x+ x\otimes 1 +w$, with $w\in R\otimes_k R$ and $\beta $ a grouplike element of $R$. We are unaware of any Hopf algebra $R$ that is a domain for which $R\otimes_k R$ is not a domain and in light of our work, it would be interesting if one could prove that $R\otimes_k R$ is a domain whenever $R$ is, when $R$ is a Hopf algebra. Using this result, along with earlier work of \cite{Gzhuang}, we obtain the following theorem.

\begin{theorem} 
Let $R$ be a noetherian Hopf $k$-algebra and let $T=R[x; \sigma, \delta]$ be an Ore extension of $R$. Suppose that $R\otimes_k R$ is a domain. Then $T$ has a Hopf algebra structure extending that of $R$ if and only if after a change of variables we have the following:
\begin{enumerate}
\item there exists a grouplike element $\beta$ of $R$ and $w\in R\otimes_k R$ such that $\Delta(x)=\beta^{-1}\otimes x+x\otimes 1 +w$, and $S(x)=-\beta(x+\sum w_1S(w_2))$ and $\beta\sum w_1S(w_2)=\sum S(w_1)w_2$;
\item there is a character $\chi : R \longrightarrow k$ such that \[ \sigma(r)=\sum\chi(r_1)r_2=\sum \beta^{-1}  r_1\beta \chi(r_2)=\textrm{ad}(\beta^{-1} )\circ \tau^r_{\chi}(r);\]
for all $r\in R$, where $\sigma$ is a left winding automorphism $\tau^l_{\chi}$, and is the composition of the corresponding right winding automorphism with conjugation by $\beta $;
\item the $\sigma$-derivation $\delta$ satisfies the relation
\[ \Delta\delta(r)-\sum \delta(r_1)\otimes r_2-\sum \beta ^{-1}r_1\otimes \delta(r_2)-w\Delta(r)-\Delta\sigma(r)w=0\]
and \[w\otimes 1+(\Delta\otimes I)(w)=\beta^{-1} \otimes w+(I\otimes \Delta)w. \]
\end{enumerate}
\label{thm:main1}
 \end{theorem}
 
In the above theorem, the hypothesis that $R$ is noetherian is needed to ensure that the antipode is bijective \cite{Skr}, which allows us to use work of \cite{Gzhuang} to get that $S(x)$ has a linear form. The hypothesis that $R\otimes_k R$ is a domain plays a more significant role. However, it appears to be difficult to show that $R\otimes_k R$ is a domain when $R$ is a Hopf algebra that is a domain. Rowen and Saltman \cite{RS} exhibit division $k$-algebras $E$ and $F$, both finite-dimensional over their centres and each containing an algebraically closed field $k$ of characteristic 0, such that $E\otimes_k F$ not a domain. Their construction is non-trivial and it does not obviously lend itself to produce a counterexample in the Hopf algebra case. In this paper, we shall show that $R\otimes_k R$ is a domain when $k$ is algebraically closed of characteristic zero and $R$ is a noetherian cocomutative Hopf algebra of finite Gelfand-Kirillov dimension that is a domain. In this case, one has that $R$ is isomorphic to the smash product of the enveloping algebra of a finite-dimensional Lie algebra $\mathcal{L}$ and a finitely generated nilpotent-by-finite group. The underlying Lie algebra $\mathcal{L}$ is generated by the primitive elements in $R$, and the nilpotent-by-finite group is just the group of grouplike elements of $R$, which acts on $\mathcal{L}$ via $k$-algebra automorphisms, giving the smash product structure. In this case, we prove the following theorem:
 \begin{theorem} Let $k$ be an algebraically closed field of characteristic zero and let $R$ be a noetherian cocommutative $k$-Hopf algebra of finite Gelfand-Kirillov dimension that is a domain. Then $R\otimes_k R$ is a domain. In particular, the results of Theorem \ref{thm:main1} apply in this setting.
\label{thm:main2}
\end{theorem}

The outline of this paper is as follows. In \S2, we prove Theorem \ref{thm:main1} and then in \S3, we give the proof of Theorem \ref{thm:main2}.  We conclude in \S4 with some pertinent remarks and questions.

\section{Proof of Theorem \ref{thm:main1}}
In this section, we give a proof of our main result. Throughout this section, we take $R$ to be a Hopf $k$-algebra and $T=R[x; \sigma, \delta]$ to be a skew polynomial extension of $R$. Suppose that $T$ admits a Hopf algebra structure extending that of $R$. Recall that $T$ is a free left $R$-module with basis $\{x^i\,\, |\,\, i\geq 0\}$ and $T\otimes_k T$ is a left $R\otimes_k R$-module with basis $\{x^i\otimes x^j \colon i,j\geq 0\}$. Thus we have that $$ \Delta(x)=\sum_{i,j}w_{i,j}x^i\otimes x^j,$$ with $w_{i,j}\in R\otimes_k R$. In fact, the hypothesis that $R\otimes_k R$ is a domain gives that 
\begin{equation}
\label{eq:Delta}
\Delta(x)=s(1\otimes x) +t(x\otimes 1)+v(x\otimes x) +w,
\end{equation} with $s, t, v, w\in R\otimes_k R$ (see \cite[Lemma 1, \S2.2]{Gzhuang}). After a change of variables and corresponding adjustments to $\delta$, we may assume that $\epsilon(x)=0.$ For if $\epsilon(x)=c\ne 0\in k$, then let $y=x-c$ and so $\epsilon(y)=0$, $yr=\sigma(r)y+\sigma(r)c+\delta(r)-cr$. Let $\delta'(r)=\delta(r)+\sigma(r)c-cr$. Then a straighforward computation shows that $\delta'(ab)=\sigma(a)\delta'(b)+\delta'(a)b$, whence $\delta'$ is a $\sigma$-derivation. Therefore, $R[x, \sigma, \delta]\cong R[y, \sigma, \delta']$.

In the following Lemmas, we will show much more: after a change of variables we have $\Delta(x)=\beta^{-1}  \otimes x+ x\otimes 1 + w$, where $\beta $ is a grouplike element of $R$. This is a significant step, as it shows that $\Delta(x)$ can be assumed to have a much simpler form, which gives an explicit Hopf algebra structure on the Ore extension $T$ that is compatible with the Hopf structure on $R$.

 To begin, we list the following facts which are useful in the proof of the subsequent Lemmas. Using coassociativity of $\Delta: T\to T\otimes T$ and the form given in Equation (\ref{eq:Delta}) and then comparing the coefficients of all relevant terms (e.g., $x\otimes
 x\otimes x$, $1\otimes x\otimes 1,$ $1\otimes x\otimes x,$ $1\otimes 1\otimes x,$ $x\otimes 1\otimes 1,$ $x\otimes x\otimes 1$) on both sides of the equation $(I\otimes \Delta)\Delta(x)=(\Delta\otimes I)\Delta(x)$, we obtain the following equations:
 \begin{align} 
 \label{te1}(I\otimes \Delta)(v)\cdot(1\otimes v)&=(\Delta\otimes I)(v)\cdot(v\otimes 1)\\ 
 \label{te2}(I\otimes \Delta)(s)\cdot(1\otimes t)&=(\Delta\otimes I)(t)\cdot(s\otimes 1)\\ 
 \label{te4}(I\otimes \Delta)(s)\cdot(1\otimes v)&=(\Delta\otimes I)(v)\cdot(s\otimes 1)\\ 
 \label{te5}(I\otimes \Delta)(s)\cdot(1\otimes s)&=(\Delta\otimes I)(s) + (\Delta\otimes I)(v)\cdot(w\otimes 1)\\ 
 \label{te7}(I\otimes \Delta)(v)\cdot(1\otimes t)&=(\Delta\otimes I)(t)\cdot(v\otimes 1).
\end{align}
We use these equations to derive additional useful equations. Throughout, we use Sweedler notation to make things more compact, that is, we simply write $f=\sum f_1\otimes f_2$ for an element $f$ of $R\otimes_k R$, with the understanding that this is actually a sum of pure tensors. We note that we may always assume, in addition, that when we choose an expression for an element $\sum_{i=1}^d a_i \otimes b_i\in R\otimes_k R$, that $\{a_1,\ldots ,a_d\}$ and $\{b_1,\ldots ,b_d\}$ are $k$-linearly independent sets. We set 
\begin{equation}\label{eq:alpha}
\alpha=(I\otimes \epsilon)(s)=\sum s_1 \epsilon(s_2),
\end{equation}
and 
\begin{equation}\label{eq:beta}
\beta=(\epsilon\otimes I)(t)=\sum \epsilon(t_1)(t_2)
\end{equation}
Observe that applying $I\otimes \epsilon \otimes \epsilon$ to Equation (\ref{te4}), we obtain on the left side
$$(I\otimes \epsilon \otimes \epsilon)((I\otimes \Delta)(s)\cdot(1\otimes v)),$$ which is 
$$\left(\sum s_1 \epsilon(s_2)\right)\cdot \left(\sum \epsilon(v_1)\epsilon(v_2)\right)=\alpha\cdot \left(\sum \epsilon(v_1)\epsilon(v_2)\right),$$
and on the right side, we obtain
$(I\otimes \epsilon \otimes \epsilon)((\Delta\otimes I)(v)\cdot(s\otimes 1))$, which is $$(\sum v_1\epsilon(v_2))\cdot \left(\sum \epsilon(s_2)s_1\right)=\left(\sum v_1\epsilon(v_2)\right)\cdot \alpha.$$
Thus we obtain the new equation
\begin{equation}
\label{alphav}
\alpha\cdot \left(\sum \epsilon(v_1)\epsilon(v_2)\right) = \left(\sum v_1\epsilon(v_2)\right)\cdot \alpha.
\end{equation}
We do not give the complete details of the following computations, as they can be done in a similar manner.
We apply $\epsilon \otimes \epsilon \otimes I$ to Equation (\ref{te1}) and we obtain
\begin{equation}
 \label{epsilonv}
 \left(\sum \epsilon(v_1)v_2\right)\cdot\left(\sum \epsilon(v_1)v_2\right)=\left(\sum \epsilon(v_1)v_2\right)\cdot\left(\sum \epsilon(v_1)\epsilon(v_2)\right).
\end{equation}
By a result of Skryabin \cite[Corollary 1]{Skr}, $S$ is bijective on $T$ and $R$ so we must have $S(x)=a x+b$ with $a,b\in R$ and $a$ a unit in $R$.
Notice that 
 \begin{equation}\label{eq:epsilonx}
  0=\epsilon(x)=m\circ(S\otimes I)\circ\Delta(x).
 \end{equation} 
 The coefficient of $x^2$ in the right side of Equation (\ref{eq:epsilonx}) is $\sum a\sigma(S(v_1)v_2)$, and the coefficient of $x^{2}$ on the left side of Equation  (\ref{eq:epsilonx}) is $0$. Since $a$ is a unit and $\sigma$ is an automorphism, we see that $\sum S(v_1)v_2=0$ and after the standard fact that $\epsilon \circ S=\epsilon$ then obtain that  \begin{equation}
  \sum \epsilon(v_1)\epsilon(v_2)=0.\label{eq:ISv}
 \end{equation} 
Now we apply the $(m\otimes I)\circ(I\otimes \epsilon \otimes I)$ to Equation  (\ref{te4}).  We obtain on the left side 
\begin{align*}
&(m\otimes I)\left((I\otimes \epsilon\otimes I)((I\otimes \Delta)(s)\cdot(1\otimes v))\right)\\
=&(m\otimes I)\left(\sum s_1\otimes (\sum \epsilon(s_{21})\epsilon(v_1)\otimes s_{22}v_2)\right)\\
=&\sum s_1\otimes s_2\sum \epsilon(v_1)v_2
\end{align*}
and on the right side
$$(m\otimes I)\circ(I\otimes \epsilon \otimes I)((\Delta\otimes I)(v)\cdot(s\otimes 1))=v(\sum s_1\epsilon(s_2))\otimes1).$$
So we can see that 
\begin{equation}
 \label{eq:sc} s\left(1\otimes (\sum \epsilon(v_1)v_2)\right) =v\left((\sum s_1\epsilon(s_2))\otimes 1\right).
 \end{equation}
\begin{lemma}\label{l1}
Let $R$ be a noetherian Hopf $k$-algebra and let $T=R[x;\sigma, \delta]$ admit a Hopf algebra structure with $R$ a Hopf subalgebra. Suppose that $R\otimes_k R$ is a domain. Then after a change of the variables with the property that $\epsilon(x)=0$ and corresponding adjustments to $\sigma$ and $\delta$, we can ensure that $v=0$ in Equation (\ref{eq:Delta}); namely, that $\Delta(x)=s(1\otimes x) +t(x\otimes 1)+w,$ with $s, t, w\in R\otimes_kR.$
\end{lemma}
 
 \begin{proof}
 Suppose $R\otimes_k R$ is a domain. As argued in the above part, we have Equation (\ref{eq:Delta}) 
 \begin{align*} 
 \Delta(x)&=s(1\otimes x) +t(x\otimes 1)+v(x\otimes x) +w,
  \end{align*}
where $s, t, v, w\in R\otimes_k R.$ 
 Using the fact that $(\epsilon\otimes I)\circ \Delta(x)=(I\otimes \epsilon)\circ \Delta(x)= x$ and that $\epsilon(x)=0$ in Equation (\ref{eq:Delta}) gives $1=(\epsilon\otimes I)(s)= (I\otimes \epsilon)(t)$; that is, 
 \begin{equation}\label{eq: identity}
 1=\sum \epsilon(s_1)s_2=\sum t_1\epsilon(t_2).
 \end{equation}
 Equations (\ref{eq:alpha}), (\ref{eq:beta}) and (\ref{eq: identity}) tell us that 
 \begin{equation*}
 1=\epsilon(\alpha)=\epsilon(\beta),
 \end{equation*} so in particular $\alpha$ and $\beta$ are nonzero.
 Thus, Equations (\ref{alphav}) and (\ref{eq:ISv}) give
  \begin{equation*}
  \sum \epsilon(v_1)\epsilon(v_2) = \sum  v_1\epsilon(v_2)=0.
 \end{equation*}
 Further, Equation (\ref{epsilonv}) tells that
 \begin{equation*}
 \sum\epsilon(v_1)v_2=\sum\epsilon(v_1)\epsilon(v_2)=\sum v_1 \epsilon(v_2)=0.
  \end{equation*}
 Thus by  Equation (\ref{eq:sc}), 
  we see that  $0=v(\alpha \otimes 1).$
 Since $\alpha \otimes 1 \ne 0$, and $R\otimes_k R$ is a domain, we see that $v=0$. 
 Thus we have shown that $\Delta(x)=s(1\otimes x) + t(x\otimes 1) +w$.
 \end{proof}
 \begin{lemma}\label{l2}
 Let $R$ be a noetherian Hopf $k$-algebra and suppose that $T=R[x;\sigma, \delta]$ admits a Hopf algebra structure with $R$ a Hopf subalgebra. Suppose that $R\otimes_k R$ is a domain and $\Delta(x)=s(1\otimes x)+t(x\otimes 1) +w$, with $s,t,w\in R\otimes_k R$. Then after a change of the variable $x$, we can assume that $\Delta(x)=\beta^{-1}\otimes x+ x\otimes 1 + w'$, where $\beta$ is a grouplike element in $R$ and $w'=\sum w'_1\otimes w'_2\in R\otimes_k R$. Moreover, $S(x)=-\beta(x+\sum w'_1S(w'_2))$ and $\beta\cdot\left(\sum w'_1S(w'_2)\right)=\sum S(w'_1)w'_2$.
 \end{lemma}
 \begin{proof}
By the assumption that $\Delta(x)$ has the form of  Equation (\ref{eq:Delta}) with $v=0$,  we get
 \begin{align} 
 \label{newte5}(I\otimes \Delta)(s)\cdot(1\otimes s)&=(\Delta\otimes I)(s)
 \end{align} from Equation (\ref{te5}).
Applying $I\otimes S\otimes I$ to Equation (\ref{newte5}), we obtain that
\[(I\otimes S\otimes I)\left( (I\otimes \Delta)(s)\cdot(1\otimes s)\right)=(I\otimes S\otimes I)\left((\Delta\otimes I)(s)\right).\]
By the associativity of the multiplication map, i.e., $m\circ(I\otimes m)=m\circ(m\otimes I): R\otimes R\otimes R\longrightarrow R$,
if we apply $m\circ (I\otimes m)=m\circ (m\otimes I)$ to both sides, then we obtain on the left side 
\begin{align*}
&m\circ (I\otimes m)((I\otimes S\otimes I)((I\otimes \Delta)(s)\cdot(1\otimes s)))\\
=&m\circ(I\otimes m)\left(\sum s_1\otimes (\sum S(s_1)S(s_{21})\otimes s_{22}s_2)\right)\\
=&\sum s_1\epsilon(s_2)\left(\sum S(s_1)s_2\right)\\
=&\alpha\left(\sum S(s_1)s_2\right)
\end{align*}
and on the right side
$$m\circ(m\otimes I)\circ(I\otimes S \otimes I)\circ(\Delta\otimes I)(s)=m\circ(\epsilon\otimes I)(s)=\sum \epsilon(s_1)s_2=1.$$
 Therefore, we have \[\alpha\left(\sum S(s_1)s_2\right)=1. \]  
 Since $R$ is a domain and $\alpha$ is left invertible, $\alpha$ is invertible and $\alpha^{-1}=\sum S(s_1)s_2$. Applying $I\otimes \epsilon \otimes I$ to Equation $(\ref{te2})$, and using Equations (\ref{eq:alpha}) and (\ref{eq:beta}), we see that 
\begin{align}\label{alphainver}
 s(1\otimes \beta)&=t(\alpha \otimes 1).
\end{align}
Note that $\alpha$ is a unit, and thus 
\begin{equation} \label{e7} 
s(\alpha^{-1}\otimes \beta)=t.
\end{equation}
Combining Equations $(\ref{te2})$ and $(\ref{e7})$, we have
\begin{equation*}
 (I\otimes \Delta)(s)\cdot(1\otimes s)\cdot(1\otimes \alpha^{-1} \otimes  \beta)=(\Delta \otimes I)(s)\cdot(\Delta (\alpha^{-1})\otimes \beta)\cdot (s\otimes 1).
\end{equation*}
By Equation (\ref{newte5}), we have 
\begin{equation}
 \label{eq:gbeta}
 (\Delta\otimes I)(s)\cdot(1\otimes \alpha^{-1} \otimes  \beta)=(\Delta \otimes I)(s)\cdot(\Delta (\alpha^{-1})\otimes \beta)\cdot (s\otimes 1).
\end{equation}
Applying $(I\otimes I \otimes \epsilon)$ to Equation (\ref{eq:gbeta}) and using the fact that $\epsilon(\beta)=1$, it results that 
\begin{equation}
   \label{eq:deltaalpha}
   \sum \epsilon(s_2)\Delta(s_1)\cdot (1\otimes \alpha^{-1})=\sum \epsilon(s_2)\Delta(s_1)\cdot \Delta (\alpha^{-1}) \cdot s.
\end{equation}
Note again that $R\otimes_k R$ is a domain and $\alpha \ne 0$.
Cancelling $\Delta(\alpha)=\sum \epsilon(s_2)\Delta(s_1)$ from both sides of Equation (\ref{eq:deltaalpha}), we have 
\begin{align}\label{deltaalphainver}1\otimes \alpha^{-1}= \Delta (\alpha^{-1}) \cdot s.
\end{align}
Then
\begin{equation*}
\begin{aligned}
\Delta(\alpha^{-1}x)&=\Delta(\alpha^{-1})\cdot\Delta(x)\\
&=\Delta(\alpha^{-1})\cdot(s(1\otimes x)+t(x\otimes 1)+w)\\
&=1\otimes \alpha^{-1}x+\Delta(\alpha^{-1})\cdot t(x\otimes 1)+ \Delta(\alpha^{-1})w\\
&=1\otimes \alpha^{-1} x+\Delta(\alpha^{-1})\cdot t(\alpha\otimes 1)(\alpha^{-1}x\otimes 1)+ \Delta(\alpha^{-1})w\\
&=1\otimes \alpha^{-1} x+\alpha^{-1}x\otimes \alpha^{-1}\beta+ \Delta(\alpha^{-1})w \,\,(\textrm{By Equations (\ref{alphainver}) and (\ref{deltaalphainver})}).
\end{aligned}
\end{equation*}
Replace $x$, $\beta$ and $w$ by $\alpha^{-1}x$, $\alpha^{-1}\beta$ and $\Delta(\alpha^{-1})w$, respectively. Then we have that 
\begin{align}\label{e8}
 \Delta(x)=1\otimes x+x\otimes \beta+ w.
\end{align}
Using the fact that $(\Delta\otimes I)\circ\Delta(x)=(I \otimes \Delta)\circ\Delta(x)$ along with Equation $(\ref{e8})$, if we compare the coefficients of  $x\otimes 1 \otimes 1$, then we obtain the equation:
$ \Delta(\beta)=\beta\otimes \beta.$
Hence $\beta$ is a grouplike element and thus has inverse. 
Notice that
\begin{equation*}
   \Delta(x\beta^{-1})=\Delta(x)\Delta(\beta^{-1})=\beta^{-1} \otimes x\beta^{-1}+x\beta^{-1}\otimes1 + w\Delta(\beta^{-1}).
\end{equation*} 
To get a simpler form of $S(x)$ later, one can replace $x$ by $x\beta^{-1}$ and 
after a change of variables,  we can assume that 
\begin{equation}
   \label{eq:newdelta}
   \Delta(x)=\beta^{-1} \otimes x+x\otimes 1+ w.
\end{equation}
Using the identity that $m\circ(I\otimes S)\circ\Delta(x)=m\circ(S\otimes I)\circ \Delta(x)=\epsilon(x)$ and Equation (\ref{eq:newdelta}), a direct computation shows that  $S(x)=-\beta(x+\sum w_1S(w_2))$ and $\beta\sum w_1S(w_2)=\sum S(w_1)w_2$.
 \end{proof}
As a consequence, we have the following corollary.
\begin{corollary}\label{c1}
Let $R$ be a noetherian Hopf $k$-algebra and suppose that $T=R[x;\sigma, \delta]$ admits a Hopf algebra structure extending that of $R$. Suppose that $R\otimes_k R$ is a domain. Then after a change of variables for the variable $x$, we have $\Delta(x)=\beta^{-1}\otimes x+ x\otimes 1 + w$, where $\beta$ is a grouplike element in $R$ and $w=\sum w_1\otimes w_2\in R\otimes_k R$ and thus condition $(iii)$ in Definition \ref{defn:BOZZ} follows from conditions $(i)$ and $(ii)$.
In particular, the Question \ref{quest} has an affirmative answer under the above hypotheses.
\end{corollary}
 This corollary allows us to immediately obtain our main result.
\begin{proof}[Proof of Theorem \ref{thm:main1}]
 Suppose that $R\otimes_k R$ is a domain. Let $T=R[x;\sigma,\delta]$ be a Hopf algebra with a Hopf structure extending that of the Hopf algebra $R$. 
 Then we have $(i)$ follows from Lemmas \ref{l1} and \ref{l2}.
 
The maps $\Delta$, $\epsilon$ and $S$ of $T$ must preserve the relation $xr=\sigma(r)x+\delta(r)$. In particular, we have the following equations:
 \begin{align*}
 &\Delta(x)\Delta(r)=\Delta(\sigma(r))\Delta(x)+\Delta(\delta(r));\\ 
 &\epsilon(x)\epsilon(r)=\epsilon(\sigma(r))\epsilon(x)+\epsilon(\delta(r));\\
 & S(r)S(x)=S(x)S(\sigma(r))+S(\delta(r)).
\end{align*}
Using arguments from \cite[Theorem 1.3]{Pa} and \cite[Theorem, \S 2.4]{Gzhuang}, we obtain $(ii)$ and $(iii)$.

Conversely, a similar argument to that used in \cite[Theorem 1.3]{Pa} and \cite[Theorem, \S 2.4]{Gzhuang} shows that $(i)$, $(ii)$, and $(iii)$ imply that $T$ is a Hopf algebra with $R$ as a Hopf subalgebra.
\end{proof}
\section{Cocommutative Hopf algebras}
In light of Theorem \ref{thm:main1}, it becomes natural to ask when $R\otimes_k R$ is a domain. Obviously, a necessary condition is that $R$ be a domain, but this is not sufficient in general, even in the case of an algebraically closed base field (see work of Rowen and Saltman \cite{RS}). We focus on the special case: when $R$ is a cocomutative noetherian Hopf algebra of finite Gelfand-Kirillov dimension over an algebraically closed field $k$ of characteristic zero that is a domain, and prove in this case that $R\otimes_k R$ is also a domain. To complete the proof of Theorem \ref{thm:main2}, we will need a result describing when crossed products are domains. The following theorem, whose proof can be found in the book of Passman \cite[Corollary 37.11]{pd}
\begin{theorem}\label{t5}
Let $R$ be an Ore domain and let let $G$ be a group and suppose that $G$ has a finite subnormal series
$$\{1\}=G_0\lhd G_1\lhd \cdots \lhd G_n=G$$ with each quotient $G_{i+1}/G_i$ locally polycyclic-by-finite. If $G$ is torsion-free then the crossed product $R\star G$ is an Ore domain. In particular if $R$ is an Ore domain and $G$ is a torsion-free polyclic-by-finite group then the smash product $R\# G$ is a domain.
\end{theorem}
Using this result, we can give the proof of Theorem \ref{thm:main2}.
 \begin{proof} [Proof of Theorem \ref{thm:main2}]
 By a refinement of a result of Kostant (see Bell and Leung \cite[Proposition 2.1]{BL14}), we have that $R\cong U(\mathcal{L}_0)\# kH$ where $\mathcal{L}_0$ is a finite-dimension Lie algebra over $k$ and $H$ is a finitely generated nilpotent-by-finite group that acts on $\mathcal{L}_0$. Hence, we have $R\otimes_k R=U(\mathcal{L}_0\oplus \mathcal{L}_0)\#k[H\times H]$. Let $\mathcal{L}$ denote the Lie algebra $\mathcal{L}_0\oplus \mathcal{L}_0$ and let $G$ denote $H\times H$.  Then $R\otimes_k R=U(\mathcal{L}) \# kG$, where $G$ acts on $U(\mathcal{L})$ in the natural way induced from the action of $H$ on $\mathcal{L}_0$. 
 
 Since $R$ is a domain, $H$ is torsion-free, and thus $G$ is also torsion-free. Moreover, $G$ is also finitely generated and nilpotent-by-finite, since $H$ is. Since $\mathcal{L}$ is finite-dimensional, we have that $U(\mathcal{L})$ is an Ore domain; moreover $G$ is a torsion-free polycyclic-by-finite group, and so we see that $R\otimes_k R$ is a domain from Theorem \ref{t5}.
 \end{proof}
 \begin{corollary}
Let $k$ be an algebraically closed field $k$ of characteristic zero and let $R$ be a noetherian cocommutative Hopf algebra of finite Gelfand-Kirillov dimension over $k$ which is a domain. Let $T=R[x; \sigma, \delta]$ be an Ore extension over $R$. Then $T$ has a Hopf algebra structure extending that of $R$ if and only if after a change of variables we have the following:
\begin{enumerate}
\item there exists a grouplike element $\beta$ of $R$ and $w\in R\otimes_k R$ such that $\Delta(x)=\beta^{-1}\otimes x+x\otimes 1 +w$, and $S(x)=-\beta(x+\sum w_1S(w_2))$ and $\beta\sum w_1S(w_2)=\sum S(w_1)w_2$;
\item there is a character $\chi : R \longrightarrow k$ such that \[ \sigma(r)=\sum\chi(r_1)r_2=\sum \beta^{-1} r_1\beta \chi(r_2)=\textrm{ad}(\beta^{-1} )\circ \tau^r_{\chi}(r);\]
for all $r\in R$, where $\sigma$ is a left winding automorphism $\tau^l_{\chi}$, and is the composition of the corresponding right winding automorphism with conjugation by $\beta $;
\item the $\sigma$-derivation $\delta$ satisfies the relation
\[ \Delta\delta(r)-\sum \delta(r_1)\otimes r_2-\sum \beta ^{-1}r_1\otimes \delta(r_2)-w\Delta(r)-\Delta\sigma(r)w=0\]
and \[w\otimes 1+(\Delta\otimes I)(w)=\beta^{-1} \otimes w+(I\otimes \Delta)w. \]
\end{enumerate}
\label{cor1}
 \end{corollary}
 \begin{proof}
 Theorem \ref{thm:main2} tells us that in this case $R\otimes_kR$ is a domain. Then
 the claim immediately follows from Theorem \ref{thm:main1}.
 \end{proof}
 \section{Concluding remarks}
 We note that in the paper \cite{Bell}, a version of Corollary \ref{cor1} was proved for finitely generated commutative Hopf algebras that are domains over an algebraically closed field $k$.  We note that this follows immediately from our Theorem \ref{thm:main1}, as such an algebra $R$ is of the form $\mathcal{O}(G)$ for $G$ an irreducible affine algebraic group over $k$ and so $R\otimes_k R$ is just $\mathcal{O}(G\times G)$, which is again a domain.  In Corollary \ref{cor1}, we have the hypothesis that the ring have finite Gelfand-Kirillov dimension.  Conjecturally, the result should hold for noetherian cocommutative Hopf algebras $R$ over an algebraically closed field of characteristic zero, since such algebras are isomorphic to algebras of the form $U(L)\# G$; since $R$ is faithfully flat over both $U(L)$ and the group algebra $k[G]$, if $R$ is noetherian, then so must these two subalgebras.  Conjecturally, enveloping algebras are noetherian if and only if $L$ is finite-dimensional and $k[G]$ is noetherian if and only if $G$ is polycyclic-by-finite.  Hence Theorem \ref{t5} can be applied to give that $R\otimes_k R\cong U(L\oplus L)\# (G\times G)$ is a domain if $R$ is a domain, since $G$ is necessarily torsion-free.  In light of this, we ask the following questions.
 \begin{question}
 Let $R$ be a cocommutative noetherian Hopf algebra over an algebraically closed field $k$ of characteristic zero.  Is $R\otimes_k R$ a domain if $R$ is a domain?
 \end{question}
 If the reader feels like being more ambitious, we raise the following question, which, combined with Theorem \ref{thm:main1}, would give an affirmative answer to Question \ref{quest} in the case when $R$ is a domain over an algebraically closed field if 
 it could be answered affirmatively.
 \begin{question}
 Let $k$ be an algebraically closed field and let $R$ be a noetherian Hopf algebra that is a domain. Is $R\otimes_k R$ a domain?
 \end{question}
 \section*{Acknowledgments} The author gratefully acknowledges her advisor Jason Bell for his constant encouragement and advice. The author also thanks Ken Brown for useful comments and thanks the referee for suggesting an improvement to the proofs of Lemmas \ref{l1} and \ref{l2}.

\end{document}